\documentclass[11pt]{article}

\usepackage[textwidth=16cm,textheight=21cm]{geometry}
\usepackage{amsmath,amssymb,amsfonts}%

\usepackage{xcolor}%
\usepackage{textcomp}%
\usepackage{manyfoot}%
\usepackage{booktabs}%

\usepackage[colorlinks,allcolors=blue]{hyperref}
\usepackage{url}
\usepackage{mathtools}

 \newcommand{\eps}{\varepsilon}

\newcommand{\Pn}{\mathrm{P}}

  \newcommand{\Tb}{\mathbf{T}}
\newcommand{\eqdef}{\coloneqq}
 \newcommand{\dd}{\mathrm{d}}

  \newcommand{\B}{\mathcal{B}}

 \newcommand{\N}{\mathbb{N}}
  \newcommand{\enumber}[1]{\operatorname{e}^{#1}}

 \newcommand{\X}{\mathcal{X}}
  \newcommand{\Disk}{\mathbb{D}}

 \newcommand{\Real}{\mathbb{R}}
 
\newcommand{\No}{\mathbb{N}_{0}}
\newcommand{\g}{\mathrm{g}}
 \newcommand{\Complex}{\mathbb{C}}
\newcommand{\Complexn}{\mathbb{C}^{n}}

 \newcommand{\Fock}{\mathcal{F}^{2}(\mathbb{C})}

 \newcommand{\RPlus}{\Real_{\geq0}}

 \newcommand{\Ele}{L_{2}}
 \newcommand{\Bergman}{\mathcal{A}^{2}}

\newcommand{\SQRTO}{\operatorname{RO}(\mathbb{N}_{0})}
\newcommand{\Lip}{\operatorname{Lip}(\mathbb{N}_{0})}

\usepackage{amsthm}%
\newtheorem{theorem}{Theorem}[section]
\newtheorem{proposition}[theorem]{Proposition}
\newtheorem{lemma}[theorem]{Lemma}
\newtheorem{corollary}[theorem]{Corollary}
\theoremstyle{remark}%
\newtheorem{example}[theorem]{Example}%
\newtheorem{remark}[theorem]{Remark}%

\theoremstyle{definition}%
\newtheorem{definition}{Definition}%

\title{Constructive approximation of convergent
sequences\\
by eigenvalue sequences of radial Toeplitz--Fock operators}

\author{Kevin Esmeral García, Egor A. Maximenko}

\begin{document}

\maketitle

\begin{center}
In the memory of Prof. Nikolai Vasilevski,\\
our guide in this area of mathematics.
\end{center}

\begin{abstract}
It is well known that for every measurable function $a$, essentially bounded on the positive halfline,
the corresponding radial Toeplitz operator $T_a$, acting in the Segal--Bargmann--Fock space, is diagonal with respect to the canonical orthonormal basis consisting of normalized monomials.
We denote by $\gamma_a$
the corresponding eigenvalues sequence.
Given an arbitrary convergent sequence,
we uniformly approximate it
by sequences of the form $\gamma_a$
with any desired precision.
We give a simple recipe for constructing $a$
in terms of Laguerre polynomials.
Previously, we proved this approximation result
with nonconstructive tools
(Esmeral and Maximenko, ``Radial Toeplitz operators on the Fock space and square-root-slowly oscillating sequences'',
Complex Anal. Oper. Theory 10, 2016).
In the present paper, we also include some properties of the sequences $\gamma_a$
and some properties of bounded sequences,
uniformly continuous with respect to the sqrt-distance on natural numbers.


\end{abstract}

\maketitle

\medskip
\subsection*{Authors' data}

Kevin Esmeral García\\
https://orcid.org/0000-0003-1147-4730\\
kevin.esmeral@ucaldas.edu.co\\
Department of Mathematics, Universidad de Caldas,
Postal code 170004, Manizales, Colombia.

\medskip\noindent
Egor A. Maximenko\\
https://orcid.org/0000-0002-1497-4338\\
egormaximenko@gmail.com\\
Escuela Superior de F\'isica y Matemáticas,
Instituto Polit\'ecnico Nacional,
Postal code 07730, Mexico City, Mexico.

\section{Introduction}

Let $\Fock$ be the Segal--Bargmann--Fock space
(see~\cite{Fock,Segal,Zhu2012Fock}),
consisting of all entire functions that are square integrable
with respect to the Gaussian measure
\[
\frac{1}{\pi}\,\enumber{-|z|^{2}} \dd\nu(z),\quad z\in\Complex,
\]
where $\nu$ is the usual Lebesgue measure on $\Complex$.
We use the short name ``Fock space'' for $\Fock$.
It is well known that 
$\Fock$ is a reproducing kernel Hilbert space,
and the reproducing kernel $K_{z}$ at the point $z$ is given by
$K_{z}(w)=e^{\overline{z} w}$.
The functions $b_n(z)=\frac{z^n}{\sqrt{n!}}$,
$n\in\No=\{0,1,2,\ldots\}$,
form an orthonormal basis of $\Fock$.
We call this basis ``the canonical basis of $\Fock$''.

        
Given $\varphi$ in $L_{\infty}(\Complex)$,
the \emph{Toeplitz operator} $\Tb_{\varphi}$ with defining symbol $\varphi$
acts on $\Fock$ by the rule $\Tb_{\varphi}f=\Pn(f\varphi)$, where $\Pn$ stays for the Bargmann projection from $\Ele(\Complex,\dd\g)$ onto $\Fock$.
Operators of this kind have been extensively studied 
\cite{Bauer-Issa,Grudsky-Vasilevski,Isra-Zhu, Zhu2012Fock},
particularly in connection
with quantum mechanics \cite{Berger-Coburn,Coburn},
harmonic analysis \cite{Abreu-Faustino,M-Englis}, etc.
The C*-algebra generated by Toeplitz operators with $L_{\infty}$-symbols is not commutative.
However, Korenblum and Zhu, Vasilevski, Grudsky, Karapetyants, Quiroga-Barranco,
and other authors
\cite{Bauer-Le,KorenblumZhu1995,GrudskyKaraVasilevski,GrudskyKaraVasilevski-2,GrudskyKaraVasilevski-3,GrudskyQuirogaVasilevski2006,VasilevskiBook} discovered some classes of defining symbols that generate commutative C*-algebras of Toeplitz operators in Bergman and Fock spaces.
In this paper,
we study Toeplitz operators $\Tb_{\varphi}$ on $\Fock$,
supposing that the defining symbol $\varphi$ is a radial function,
i.e., there exists $g$ in $L_\infty(\RPlus)$,
$\RPlus\eqdef[0,+\infty)$,
such that $\varphi(z)=g(|z|)$ a.e. in $\Complex$.
Grudsky and Vasilevski \cite{Grudsky-Vasilevski} noticed that such operators are diagonal with respect to the canonical basis:
$\Tb_{\varphi}b_n = \gamma_g(n) b_n$, where
the eigenvalue sequence $\gamma_\varphi$ is given by
\begin{equation}
\label{gamma-radial-fock}
\gamma_{g}(n)
=
\frac{1}{n!}\int_{0}^{+\infty}
g(\sqrt{r}\,)\enumber{-r}r^{n}\,\dd r\quad (n\in\No).
\end{equation}
Here $\varphi(z)=g(|z|)$ a.e. in $\Complex$.

In 2008, Su\'{a}rez, in a wonderful research \cite{SuarezD}, studied the properties of eigenvalue sequences of radial Toeplitz operators acting on the Bergman space over the unit disk using the concepts of the $n$-Berezin transformation and the invariant Laplacian of a bounded operator, and introduced the space of sequences $d_{1}$:
\begin{align}\label{eq:d1}
d_1\eqdef
\left\{x \in \ell_{\infty}(\No)\colon\sup _{n\in\No}\left((n+1)\left|x_{n+1}-x_{n}\right|\,\right)<\infty\right\}.
\end{align}
One of the main results of Suárez states that the C*-algebra generated by radial Toeplitz operators with bounded generated symbols is isometrically isomorphic to the C*-algebra generated by $d_{1}$, this last being the topological closure of $d_{1}$ in $\ell_{\infty}(\No).$
Vasilevski wanted to complement Suárez's work with an explicit description of the closure of $d_{1}$.
In 2013, Grudsky, Maximenko, and Vasilevski~\cite{GrudskyMaximenkoVasilevski2013}
showed that the set
$\left\{\gamma_{a}\colon\,a\in\,L_{\infty}(0,1)\right\}$
and the space $d_{1}$ introduced by Su\'{a}rez~\cite{SuarezD}
are dense in $\operatorname{SO}(\No)$,
where $\operatorname{SO}(\No)$
is the C*-algebra of bounded sequences slowly oscillating with respect to the \emph{logarithmic distance}
$\displaystyle \eta(j,k)\eqdef\bigl|\ln(j+1)-\ln(k+1)\bigr|$.
Hence, the C*-algebra generated by radial Toeplitz operators acting on the Bergman space $\Bergman(\Disk)$
is isometrically isomorphic to $\operatorname{SO}(\No)$.

In 2012, Vasilevski suggested to us and to other colleagues to do something similar for the other cases where the C*-algebra generated by the Toeplitz operators is commutative in Bergman  and Fock spaces. This task was solved in a series of papers \cite{Esmeral-Maximenko-2014,Esmeral-Maximenko-2016,Esmeral-Maximenko-Vasilevski-2015,Esmeral-Vasilevski-2016,Herrera-Maximenko-Hutnik,Crispin-Vasilevski-Maximenko}.

In \cite{Esmeral-Maximenko-2016},
we studied the case of radial Toeplitz operators
acting on the Fock space.
Let $\gamma(L_{\infty}(\RPlus))$
be the set (vector space) of all eigenvalue sequences $\gamma_g$
given by~\eqref{gamma-radial-fock},
with essentially bounded defining symbols $g$:
\begin{equation}
\label{eq:set_gamma_L_infty}
\gamma(L_{\infty}(\RPlus))
\eqdef \bigl\{\gamma_g\colon\ g\in L_\infty(\RPlus)\bigr\}.
\end{equation}
Furthermore, let $\SQRTO$ be the C*-algebra
of all bounded sequences
$\sigma\colon\No\to\Complex$
that are uniformly continuous
with respect to the \emph{sqrt-distance}
\eqref{square-root-metric}.
In \cite{Esmeral-Maximenko-2016},
we showed that $\gamma(L_{\infty}(\RPlus))$ is a dense subset of $\SQRTO$.
As a corollary, we concluded that the C*-algebra generated by Toeplitz operators with radial $L_{\infty}$-symbols
is isometrically isomorphic to $\SQRTO$.
That proof from \cite{Esmeral-Maximenko-2016} consists of two parts.
The first one is constructive and shows that after the change of variables $\sqrt{n}=x$,
the eigenvalue sequences of the Toeplitz operators, viewed as a function $x\mapsto\gamma_{g}(x^{2})$,
for $x$ sufficiently large,
is close to the convolution of the symbol $g$
with the heat kernel.
The second one, \cite[Theorem 7.3]{Esmeral-Maximenko-2016},
is non-constructive.
By using the Hahn--Banach theorem, it shows that $c_{0}(\No)$ coincides with the uniform closure of the set
$\gamma(\X)=\{\gamma_{g}\colon\ g\in\X\}$, where 
\begin{equation}\label{eq:X}
    \X=\left\{g\in\,L_{\infty}\big(\RPlus\big)\colon\ \lim_{x\to+\infty}g(x)=0\right\}.
\end{equation}
More precisely, $\X$ consists of the equivalence classes of essentially bounded functions defined on $\RPlus$ that have limit $0$ at infinity.

Our main result from~\cite{Esmeral-Maximenko-2016}
was extended by  Dewage and  \'{O}lafsson~\cite{Dewage-Olafsson}
to the case of $k$-quasi-radial symbols acting on
$\mathcal{F}^{2}(\Complexn)$.
Recently, Dewage and  Mitkovski \cite{dewage2024} used tools from the Laplacian of a bounded operator and quantum harmonic analysis to show that the domain of the Laplacian of an operator is dense in the Toeplitz algebra over  $\Fock$, and as a by-product they revisited the Gelfand theory of the radial Toeplitz algebra on $\mathcal{F}^{2}(\Complexn)$, and provide a short proof of the main result in \cite{Esmeral-Maximenko-2016}. However, such a proof requires very advanced knowledge of quantum harmonic analysis \cite{Fulsche}. 
 
In this paper, we improve \cite{Esmeral-Maximenko-2016} in the following manner.
Let $C_0(\RPlus)$ be the set of all bounded continuous functions $f\in C_b(\RPlus)$ tending to $0$ at $+\infty$:
\begin{equation}\label{Def:C0}
C_0(\RPlus)
\eqdef
\Bigl\{f\in C_b(\RPlus)\colon
\lim_{x\to+\infty}f(x)=0
\Bigr\}.
\end{equation}
Notice that $C_0(\RPlus)$
is a closed subspace of $C_b(\RPlus)$.
Moreover, $C_0(\RPlus)$ is a subset of $\X$.
Let $\gamma(C_{0}(\RPlus))$ be the set of the eigenvalue sequences
corresponding to the defining symbols of class $C_0(\RPlus)$:
\begin{equation}\label{eq:def-Gamma0}
\gamma(C_{0}(\RPlus))
\eqdef
\left\{\gamma_{g}\in c_{0}(\No)\colon g\in C_{0}(\RPlus)\right\}.
\end{equation}
We give a constructive proof
that $c_0(\No)$ coincides with the uniform closure of $\gamma(C_{0}(\RPlus))$.
Namely,
given $\sigma$ in $c_{0}(\No)$,
we construct a sequence $(u_\xi)_{m\in\N}$
of defining symbols of class $C_0(\RPlus)$,
such that 
$(\gamma_{u_{\xi}})_{\xi\in\N}$
converges to $\sigma$ uniformly.
Our recipe for $u_\xi$ is very direct and simple;
it is given in terms of Laguerre polynomials.

The paper is organized as follows.
In Section \ref{section-sqrto-lip},
we prove that $\SQRTO$
is an invariant subspace of the left and right shift operators.
In addition, we study the space $\Lip$ of the bounded sequences that are Lipschitz continuous
with respect to sqrt-metric $\rho$. 

In Section~\ref{sec:gammas_belong_to_RO},
we study some properties of $\gamma_g$.
In particular, we prove that the space
$\gamma(L_{\infty}(\RPlus))$
is an invariant subspace of the left and right shift operators.

Section \ref{Laguerre+radial-T}
is devoted to prove our main approximation results
(Theorems~\ref{thm-gamma-delta},
\ref{thm:gamma-dense-c0},
and~\ref{thm:gamma-dense-c}).

We hope that the results and ideas of this paper will be useful in future investigations related to radial operators on Fock or Bergman spaces.

In particular, it would be interesting to find a simple formula to approximate sequences of the class $\SQRTO$ by sequences of the form $\gamma_g$, where $g\in L_\infty(\RPlus)$.

In this paper, we do not consider Toeplitz operators and eigenvalue sequences induced by non-bounded defining symbols;
that problem is discussed in~\cite{Grudsky-Vasilevski,Ramirez-Rossini-Sanmartino}.

\section{Uniformly continuous sequences with respect to the sqrt-distance}
\label{section-sqrto-lip}

In this section, we develop some ideas introduced in our previous paper \cite{Esmeral-Maximenko-2016}. 

Let $\rho\colon\No\times\No\rightarrow\RPlus$ be the square-root metric (distance) on $\No$ given by
\begin{equation}
\label{square-root-metric}
\rho(m,n)\eqdef\left|\sqrt{m}-\sqrt{n}\,\right|.
\end{equation}
Given a complex sequence $\sigma = (\sigma(j))_{j \in \No}$,
its \emph{modulus of continuity}
(or \emph{indicator of uniform continuity})
$\omega_{\rho,\sigma}\colon (0,+\infty) \to [0,+\infty]$ with respect to the sqrt-metric \eqref{square-root-metric}  by
\[
\omega_{\rho,\sigma}(\delta)
\eqdef
\sup\, \bigl\{ |\sigma(j) - \sigma(k)|\colon\ 
j,k \in \No,\ \rho(j,k) \leq \delta \bigl\}.
\]
Note that for every sequence $\sigma$, the function $\omega_{\rho,\sigma}$ is increasing (in the nonstrict sense). Therefore, the condition $\displaystyle\lim_{\delta \to 0^+} \omega_{\rho,\sigma}(\delta) = 0$ is equivalent to the following one: for every $\varepsilon > 0$, there exists $\delta > 0$ such that $\omega_{\rho,\sigma}(\delta) < \varepsilon$. 

 We denote by $\SQRTO$ the set of the bounded sequences
that are uniformly continuous
with respect to the sqrt-metric:
\begin{equation}
\SQRTO =\left\{\sigma\in\ell_{\infty}(\No)\colon
\lim_{\delta\to 0^{+}}\omega_{\rho,\sigma}(\delta)=0\right\}.
\end{equation}
A more precise notation for this class of sequences would be $\operatorname{BUC}(\No,\rho,\Complex)$. Nevertheless,
in this paper, we prefer using the short notation $\SQRTO$.

\subsection{Lipschitz continuous sequences with respect to the sqrt-metric}
According to general definitions
of Lipschitz continuous functions between metric spaces~\cite[Subsection 1.3.2]{Cobzas-Miculescu-Nicolae},
a sequence $\sigma\colon\No\to\Complex$ is called
\emph{Lipschitz continuous} with respect to $\rho$, if
\begin{equation}
\label{eq:Lip_def}
\sup_{\substack{j,k\in\No\\j\ne k}}
\frac{|\sigma(j)-\sigma(k)|}{\rho(j,k)}
<+\infty.
\end{equation}
It is obvious that 
$\sigma$ is Lipschitz continuous
if and only if
\begin{equation}
\label{eq:Lip_via_omega}
\exists L\in[0,+\infty)\qquad
\forall \delta>0\qquad
\omega_{\rho,\sigma}(\delta) \le L\,\delta.
\end{equation}

\begin{definition}
We denote by $\operatorname{BLip}(\No, \rho,\Complex)$
or shortly by $\Lip$ the set of all bounded sequences that are Lipschitz continuous with respect to $\rho$. 
\end{definition}

\begin{proposition}
\label{Prop:lip-acot}
Let $\sigma\colon\No\to\Complex$.
Then $\sigma$ is Lipschitz continuous with respect to $\rho$, if and only if,
\begin{equation}
\label{eq:Lip_criterion}
\sup_{n\in\No}
\left(\sqrt{n+1}\,\big|\sigma(n+1)-\sigma(n)\big|\right)
<+\infty.
\end{equation}
\end{proposition}

\begin{proof}
See~\cite[Proposition 3.2]{Esmeral-Maximenko-2016}.
\end{proof}

The condition \eqref{eq:Lip_criterion} is similar to the condition in the formula \eqref{eq:d1} which defines Lipschitz continuous sequences with respect to the logarithmic distance.

It follows directly from the definitions (and it is well known in the general context of metric spaces)
that Lipschitz continuous sequences are uniformly continuous with respect to $\rho$.
Therefore, $\Lip\subseteq\SQRTO$.
In the following example,
we show that this inclusion is strict.

\begin{example}
\label{example:Lip_ne_SQRTO}
Define $\sigma\colon\No\to\Complex$ by
\[
\sigma(j)
\eqdef
\sqrt{\bigl|\sin\bigl(\pi\sqrt{j}\bigr)\bigr|}.
\]
Obviously, $\sigma$ is a bounded sequence.
Using the elementary inequality
\[
|x-y|^2\le|x^2-y^2|\qquad(x,y\in\Complex)
\]
and the Lipschitz property of $\sin$,
we get the following upper estimate for every $j,k$ in $\No$:
\begin{align*}
\bigl|\sigma(j)-\sigma(k)\bigr|^{2}
&\le
\bigl|\sigma(j)^2-\sigma(k)^2\bigr|
=\left|\sin\bigl(\pi\sqrt{j}\bigr)
-\sin\bigl(\pi\sqrt{k}\bigr)\right|
\\
&\le \pi \bigl|\sqrt{j}-\sqrt{k}\,\bigr|=\pi\rho(j,k).
\end{align*}
Therefore,
$\displaystyle\bigl|\sigma(j)-\sigma(k)\bigr|\leq \sqrt{\pi\rho(j,k)}$.
This means that $\sigma$ is $\frac{1}{2}$-H\"{o}lder continuous with respect to $\rho$.
In particular, $\sigma\in\SQRTO$. 

To prove that $\sigma\notin\Lip$,
we consider the expression from
Proposition~\ref{Prop:lip-acot}
for the sequence of indices $\nu(m)=m^2$.
Notice that
\[
\sigma(\nu(m))
=\sqrt{\bigl|\sin(\pi m)\bigr|}=0.
\]
Furthermore, since
\[
\pi\sqrt{m^2+1}
=\pi m+\pi\bigl(\sqrt{m^2+1}-m\bigr)
=\pi m+\frac{\pi}{\sqrt{m^2+1}+m},
\]
the subsequence $\bigl(\sigma(\nu(m)+1)\bigr)_{m\in\No}$
has the following asymptotic behavior as $m\to\infty$:
\[
\sigma(\nu(m)+1)
=\sqrt{\bigl|\sin\bigl(\pi\sqrt{m^2+1}\bigr)\bigr|}
=\sqrt{\frac{\pi}{\sqrt{m^2+1}+m}}
\sim \sqrt{\frac{\pi}{2m}}.
\]
Therefore,
\[
\lim_{m\to\infty}
\Bigl(\sqrt{\nu(m)+1} \,\bigl|\sigma(\nu(m)+1)-\sigma(\nu(m))\bigr|
=
\lim_{m\to\infty}
\sqrt{\frac{\pi(m^2+1)}{2m}}
=+\infty.
\]
By~Proposition~\ref{Prop:lip-acot},
$\sigma$ is not Lipschitz continuous with respect to $\rho$.
\end{example}

Let us give an example of a sequence belonging to $\Lip$ and $\SQRTO$.

\begin{example}
\label{ex:sqrt_cos}
Let $\sigma=\bigl(\cos\bigl(\sqrt{j}\bigr)\bigr)_{j\in\No}$.
For every $j,k$ in $\No$,
\[
\bigl|\sigma(j)-\sigma(k)\bigr|
=
\bigl|\cos\bigl(\sqrt{j}\bigr)-\cos\bigl(\sqrt{k}\bigr)\bigr|
\le
\bigl|\sqrt{j}-\sqrt{k}\,\bigr|
=
\rho(j,k).
\]
Therefore, $\sigma\in\Lip$
and $\sigma\in\SQRTO$.
\end{example}

The following lemma means that the sqrt-distance $\rho(j,k)$ can take small nonzero values only for large arguments $j$ and $k$.

\begin{lemma}
\label{lem:min_jk_ge_rhojk}
Let $j,k\in\No$
such that $j\ne k$ and
$\rho(j,k)<1/2$.
Then,
\begin{equation}
\label{eq:min_jk_ge_rhojk}
\min\{j,k\}
\ge \frac{1}{(2\rho(j,k))^2}-1.
\end{equation}
\end{lemma}

\begin{proof}
Suppose that $j<k$ (the case when $j>k$ is similar).
Then,
\[
\rho(j,k)
=\sqrt{k}-\sqrt{j}
\ge \sqrt{j+1}-\sqrt{j}
=\frac{1}{\sqrt{j+1}+\sqrt{j}}
\ge \frac{1}{2\sqrt{j+1}}.
\]
Solving this equality for $j$,
we get the lower estimate~\eqref{eq:min_jk_ge_rhojk}.
\end{proof}

\begin{proposition}
    $c(\No)$ is a proper subset of $\SQRTO$:  $c(\No)\subsetneq\SQRTO$.
\end{proposition}

\begin{proof}
Let $\sigma\in c(\No)$.
Then, $\sigma$ is a Cauchy sequence.
Given $\eps>0$,
we choose $N$ in $\No$ such that
for all $m$ and $n$ with $m>n\ge N$,
\begin{equation}
\label{eq:difference_sigma_less_half_eps}
\bigl|\sigma(n)-\sigma(m)\bigr|<\frac{\eps}{2}.
\end{equation}
Let
$\delta\eqdef\frac{1}{2\sqrt{N+1}}$.
If $m>n$ and $\rho(m,n)\le\delta$,
then, by Lemma~\ref{lem:min_jk_ge_rhojk},
we get $n\ge \frac{1}{(2\delta)^2}-1=N$,
which implies~\eqref{eq:difference_sigma_less_half_eps}.
Therefore,
$\bigl|\sigma(m)-\sigma(n)\bigr|<\eps/2$.
We have proven that
\[
\omega_{\rho,\sigma}(\delta)
\le\frac{\eps}{2}
<\eps.
\]
Therefore, $\sigma\in\SQRTO$.
This means that $c(\No)\subseteq\SQRTO$.
On the other hand,
the sequence considered in Example~\ref{ex:sqrt_cos}
belongs to $\SQRTO\setminus c(\No)$.
\end{proof}

In the context of a general metric space $(X,d)$,
it is well known that every bounded and uniformly continuous complex function on $X$ can be uniformly approximated
by Lipschitz continuous functions;
see~\cite[Corollary 6.3.2]{Cobzas-Miculescu-Nicolae}.
Therefore,
$\overline{\Lip}= \SQRTO$.
Next, we provide a more elementary proof of this fact
using averaging techniques that are analogs
of Vall\'{e}e--Poussin means, in our context.
A similar idea was proposed in~\cite[Lemma 4.4]{GrudskyMaximenkoVasilevski2013}
for the logarithmic distance in $\No$.
Such techniques are widely used in the approximation theory~\cite{Timan}.

\begin{lemma}
\label{lemma-aprox-vallee-poussin-mean}
Let $\sigma \in \ell_{\infty}(\No)$ and $\delta \in(0,1)$.
Define $y\colon\No\to\Complex$ by
\begin{equation}\label{eq:vallee-poussin-mean}
    y(j)=\frac{1}{1+r_{j}} \sum_{k=j}^{j+r_{j}} \sigma(k),
    \quad \text{where}\quad
    r_{j}=\lfloor \delta \sqrt{j}\rfloor.
\end{equation}
Then, $y \in \Lip$ and
\begin{equation}
\label{eq:lemma44}
\|y-\sigma\|_{\infty} \leq \omega_{\rho, \sigma}(\delta) .
\end{equation}
\end{lemma}

\begin{proof}
For every $j$ in $\No$,
the sum in the right-hand side of~\eqref{eq:vallee-poussin-mean} contains $1+r_{j}$ terms.
Therefore,
$\|y\|_\infty \le \|\sigma\|_{\infty}$.

Let $j\in\No$.
Since $\delta(\sqrt{j+1}-\sqrt{j})<1$,
it is easy to see that $r_{j+1}=r_j$ or $r_{j+1}=r_j+1$.
For the case $r_{j}=r_{j+1}$, we have
\begin{equation}
\label{eq:rj=rj+1}
|y(j+1)-y(j)|
=
\dfrac{|\sigma(j+1+r_{j})-\sigma(j)|}{1+r_{j}}
\leq
\dfrac{2\|\sigma\|_{\infty}}{1+r_{j}}.
\end{equation}
Otherwise, $r_{j+1}=r_j + 1$, and
\[
y(j+1)
=\frac{1}{2+r_j}\sum_{k=j}^{j+r_j}\sigma(k)
+ \frac{\sigma(j+r_j+1)+\sigma(j+2+r_j)-\sigma(j)}{2+r_j},
\]
and
\[
|y(j+1)-y(j)|
\le
\left(\frac{1}{1+r_j}-\frac{1}{2+r_j}\right)
\sum_{k=j}^{j+r_j}|\sigma(k)|
+\frac{3\|\sigma\|_{\infty}}{2+r_j},
\]
which implies
\begin{equation}
\label{eq:rj+1=rj+1}
|y(j+1)-y(j)| \le \frac{4\|\sigma\|_\infty}{1+r_j}.
\end{equation}
Notice that
$1+r_j=1+\lfloor \delta\sqrt{j}\rfloor
\ge \delta \sqrt{j}$.
By \eqref{eq:rj=rj+1} and \eqref{eq:rj+1=rj+1},
for any $j$ in $\No$,
\[
\sqrt{j+1}\,|y(j+1)-y(j)|
\leq
4 \|\sigma\|_{\infty}
\dfrac{\sqrt{j+1}}{\delta\,\sqrt{j}}
\leq
\dfrac{4\sqrt{2}\,\|\sigma\|_{\infty}}{\delta}.
\]
The last expression does not depend on $j$.
According to Proposition~\ref{Prop:lip-acot},
we conclude that $y \in \Lip$.

Now, let us prove \eqref{eq:lemma44}.
If $j \leq k \leq j+r_j$, then
\[
\rho(j, k)
\leq
\sqrt{j+r_j}-\sqrt{j}
=
\dfrac{r_j}{\sqrt{j+r_j}+\sqrt{j}}
\leq
\dfrac{\lfloor \delta\sqrt{j}\rfloor}{\sqrt{j}}\leq\delta .
\]
Therefore,
\[
|y(j)-\sigma(j)|
\leq
\frac{1}{1+r_j}
\sum_{k=j}^{j+r_j}
|\sigma(k)-\sigma(j)|
\leq \omega_{\rho, \sigma}(\delta).
\qedhere
\]
\end{proof}

\begin{proposition}
\label{prop:Lip-dense-SQRTO}
$\Lip$ is a dense subset of $\SQRTO$.
\end{proposition}

\begin{proof}
Let $\eps>0$ and $\sigma\in\SQRTO$.
Using the fact that $\omega_{\rho, \sigma}(\delta) \to 0$ as $\delta \to 0$,
we choose a $\delta>0$ such that
$\omega_{\rho, \sigma}(\delta)<\eps$,
and define $y$ by~\eqref{eq:vallee-poussin-mean}.
Then, by Lemma~\ref{lemma-aprox-vallee-poussin-mean},
$y \in \Lip$ and $\|\sigma-y\|_{\infty}<\eps$.
\end{proof}

\begin{definition}[shift operators in $\ell_\infty(\No)$]
We denote by $\tau_L$ and $\tau_R$ the left and right shift operators acting on $\ell_\infty(\No)$ by the following rules.
For $\sigma$ in $\ell_\infty(\No)$,
\begin{equation}
\label{eq:tauL_tauR}
\tau_{L}\sigma
\eqdef
\left(\sigma(1), \sigma(2), \sigma(3), \ldots\right),
\quad 
\tau_{R}\sigma
\eqdef\left(0, \sigma(0), \sigma(1), \ldots\right).
\end{equation}
More formally, for each $j$ in $\No$,
\begin{align}
(\tau_{L} \sigma)(j) \eqdef \sigma(j+1);
\qquad
(\tau_{R}\sigma)(j) \eqdef
\begin{cases}
0, & j=0; \\
\sigma(j-1), & j\in\N.
\end{cases}
\end{align}
\end{definition}

\noindent
Note that $\tau_{L}\left(\tau_{R}\sigma\right)=\sigma$ for every sequence $\sigma$.

\begin{lemma}
\label{lemm:rho-ineq}
For every $j, k \in \No$,
\begin{equation}
\label{eq:lemarho-ineq}
\dfrac{1}{\sqrt{6}}\,\rho(j, k)
\leq \rho(j+1, k+1)
\leq \rho(j,k).
\end{equation}
\end{lemma}

\begin{proof}
Let $j, k \in \No$.
By the inequality of arithmetic and geometric means,
\[
j+k+2\geq 2\sqrt{(j+1)\cdot (k+1)}.
\]
Therefore,
\begin{align*}
\rho(j+1, k+1)
&=
\rho(j,k)\,
\dfrac{\left(\sqrt{j}+\sqrt{k}\,\right)}{\left(\sqrt{k+1}+\sqrt{j+1}\,\right)}
\\
&=
\rho(j,k)\,
\sqrt{\dfrac{k+j+2\sqrt{kj}}{k+j+2+2\sqrt{(k+1)(j+1)}}}
\\
&\geq
\rho(j,k)\,
\sqrt{\dfrac{k+j}{k+j+2+2\sqrt{(k+1)(j+1)}}}
\\
&\geq
\rho(j,k)\,
\sqrt{\dfrac{k+j}{2(k+j+2)}}.
\end{align*}
It is easily seen that the function $f(t)=\dfrac{t}{t+2}$, $t\geq1$
is strictly increasing and attains its minimum value $1/3$
at the point $1$.
Therefore, \eqref{eq:lemarho-ineq} holds.
\end{proof}

Next, we show that $\SQRTO$ and $\Lip$ are invariant subspaces of $\tau_L$ and $\tau_R$.
This fact is similar to
\cite[Proposition 3.10]{GrudskyMaximenkoVasilevski2013}.

\begin{proposition}\label{prop:sigma-TauL-sqrto}
$\sigma \in \SQRTO$ if and only if $\tau_{L}\sigma\in \SQRTO$.   Moreover,
$\sigma \in \Lip$ if and only if $ \tau_{L}\sigma\in \Lip$.
\end{proposition}

 \begin{proof}
Let $\sigma \colon\No\to\Complex$.
Then, we have the following connection between
the images of $\tau_{L}\sigma$ and  $\sigma$:
\[
\sigma(\No)=\tau_{L}(\No)\cup \{\sigma(0)\}.
\]
Therefore,
\begin{equation}
\label{eq:norm-sigm-tausig}
\|\sigma\|_{\infty}
=\max\left\{\left\|\tau_{L}\sigma\right\|_{\infty},\bigl|\sigma(0)\bigr|\right\}.
\end{equation}
Let $\delta>0$. Then, 
\begin{align*}
\omega_{\rho,\tau_{L}\sigma}(\delta) &= \sup \bigl\{ \big|\sigma(j+1) - \sigma(k+1)\big|\colon j,k \in \No,\ \rho(j,k) \leq \delta \bigl\}.
\end{align*}
By Lemma~\ref{lemm:rho-ineq},
$\rho(j,k) \leq \delta $ implies $\rho(j+1,k+1) \leq \delta$.
Thus, 
\begin{align*}
\omega_{\rho,\tau_{L}\sigma}(\delta)
&\leq
\sup \bigl\{ |\sigma(j+1) - \sigma(k+1)|\colon\ 
j,k \in \No,\ \rho(j+1,k+1) \leq \,\delta \bigl\}
\\
&=
\sup \bigl\{ |\sigma(p) - \sigma(q)|\colon\ 
p,q \in \N,\ \rho(p,q) \leq \,\delta \bigl\}
\\
&\leq
\omega_{\rho,\sigma}(\delta).
\end{align*}
On the other hand, by Lemma~\ref{lemm:rho-ineq},
$\rho(p+1,q+1) \leq \delta$ implies
$\rho(p,q) \leq \sqrt{6}\,\delta$.
Therefore,
\begin{align*}
\omega_{\rho,\sigma}(\delta)
&=\sup \bigl\{ |\sigma(j) - \sigma(k)|\colon\ 
j,k \in \No,\ \rho(j,k) \leq \delta \bigl\}
\\
&\leq
\sup \bigl\{ |\sigma(p+1) - \sigma(q+1)|\colon\ 
p,q \in \No,\ \rho(p+1,q+1) \leq \delta \bigl\}
\\
&\leq
\sup \bigl\{ |\sigma(p+1) - \sigma(q+1)|\colon\ 
p,q \in \No,\ \rho(p,q) \leq \sqrt{6}\,\delta \bigl\}
\\
&=
\omega_{\rho,\tau_{L}\sigma}(\sqrt{6}\,\delta).
\end{align*}
We have proved that
\begin{equation}
\label{eq:ineq-wtaul-wsig}
\omega_{\rho, \sigma}\left(\dfrac{\delta}{\sqrt{6}}\right)
\leq \omega_{\rho, \tau_{L}\sigma}(\delta)
\leq \omega_{\rho, \sigma}(\delta).
\end{equation}
The conclusions of the proposition follow by \eqref{eq:norm-sigm-tausig} and \eqref{eq:ineq-wtaul-wsig}.
We use the criterion~\eqref{eq:Lip_via_omega}
for the Lipschitz continuity.
 \end{proof}
 
\begin{proposition}
\label{prop:sigma_and_tauR_sigma}
Let $\sigma\in\ell_{\infty}(\No)$.
Then, $\sigma \in \SQRTO$ if and only if
$\tau_{R}\sigma \in \SQRTO$.
Moreover,
Then, $\sigma \in \Lip$ if and only if
$\tau_{R}\sigma \in \Lip$.
\end{proposition}

\begin{proof}
Let $\sigma\colon\No\to\Complex$.
The sequences $\sigma$ and $\tau_{R}\sigma$ have the same image up to the number zero:
\[
(\tau_R \sigma)(\No) = \sigma(\No)\cup\{0\}.
\]
Therefore,
\begin{equation}
\label{eq:norm_tau_R_sigma_and_sigma}
\|\tau_{R}\sigma\|_{\infty}=\|\sigma\|_{\infty}.
\end{equation}
Let $\delta>0$.
Applying~\eqref{eq:ineq-wtaul-wsig}
with $\tau_R \sigma$ instead of $\sigma$
and using the formula $\sigma=\tau_L(\tau_R \sigma)$,
we get
\[
\omega_{\rho,\sigma}(\delta)
\le
\omega_{\rho, \tau_R \sigma}(\delta).
\]
On the other hand,
let $0<\delta<1/3$
and $j,k\in\No$ such that $j\ne k$
and $\rho(j,k)\le\delta$.
Then, $j,k\ge 1$ by Lemma~\ref{lem:min_jk_ge_rhojk}
and $\rho(j-1,k-1)\le \sqrt{6}\,\rho(j,k)$
by Lemma~\ref{lemm:rho-ineq}.
Hence,
\begin{align*}
\omega_{\rho,\tau_R \sigma}(\delta)
&\le
\sup \bigl\{ |\sigma(j-1) - \sigma(k-1)|\colon\ 
j,k \in \N,\ \rho(j,k) \leq \delta \bigl\}
\\
&\le
\sup \bigl\{ |\sigma(j-1) - \sigma(k-1)|\colon\ 
j,k \in \N,\ \rho(j-1,k-1) \leq \sqrt{6}\,\delta \bigl\}
\\
&\le
\sup \bigl\{ |\sigma(p) - \sigma(q)|\colon\ 
p,q \in \No,\ \rho(p,q) \leq \delta \bigl\}
=\omega_{\rho,\sigma}(\delta).
\end{align*}
For $\delta\ge1/3$, we have the trivial inequality
\[
\omega_{\rho,\tau_R \sigma}(\delta)
\le 2\|\sigma\|_\infty
\le 6\|\sigma\|_\infty\,\delta.
\]
So, for every $\delta>0$,
\begin{equation}
\label{eq:omega_tau_R_sigma_and_omega_sigma}
\omega_{\rho,\sigma}(\delta)
\le
\omega_{\rho, \tau_R \sigma}(\delta)
\le
\max\bigl\{\omega_{\rho,\sigma}(\delta),
6\|\sigma\|_\infty\,\delta\bigr\}.
\end{equation}
The conclusions of the proposition
follow from
\eqref{eq:norm_tau_R_sigma_and_sigma}
and 
\eqref{eq:omega_tau_R_sigma_and_omega_sigma}.
\end{proof}

The following proposition shows a curious property of the sequences in $\SQRTO$: the difference of such sequence $\sigma$ and its shifted version $\tau_{L}^{k}\sigma$ is a zero convergent sequence.

\begin{proposition}
\label{eq:sequence_minus_its_shift}
    Let $\sigma\in\SQRTO$ and $k\in\N$. Then $\sigma-\tau_{L}^{k}\sigma\in c_{0}(\No)$.
\end{proposition}

\begin{proof}
Indeed,
\begin{align*}\left|\sigma(n)-(\tau_{L}^{k}\sigma)(n)\right|=\left|\sigma(n)-\sigma(n+k)\right|\leq \omega_{\rho,\sigma}\left(\dfrac{k}{\sqrt{n+k}+\sqrt{n}}\right),
\end{align*}
and the last expression converges to $0$ as $n$ tends to $\infty$.
\end{proof}


\section{Shifts and eigenvalue sequences of radial Toeplitz operators}
\label{sec:gammas_belong_to_RO}
Given  $S\subseteq L_{\infty}(\RPlus) $, we denote by $\gamma(S)$ the set of the corresponding eigenvalue sequences:
\begin{equation}
    \gamma(S)\eqdef\{\gamma_{g}\colon g\in S\}.
\end{equation}
We use this notation for $S=L_{\infty}(\RPlus)$, $S=C_{0}(\RPlus)$, etc.

In \cite{Esmeral-Maximenko-2016},
we showed some properties of sequences in the class $\gamma(L_{\infty}(\RPlus))$.
In particular, we showed that these sequences are Lipschitz continuous
with respect to $\rho$
\cite[Proposition 4.4]{Esmeral-Maximenko-2016},
and form a dense subspace of $\SQRTO$,
\cite[Theorem 1.1]{Esmeral-Maximenko-2016}.
Next, we show that this set is invariant under the left translations.
Our main tool is the following sequence
$(\B_j)_{j\in\No}$
of \emph{averaging operators},
introduced by Grudsky and Vasilevski in~\cite{Grudsky-Vasilevski}.
Define $\B_0\colon L_\infty(\RPlus)\to L_\infty(\RPlus)$ by
\begin{equation}
\label{eq:B0_def}
(\B_0 g)(r)\eqdef g(\sqrt{r})
\qquad(g\in L_\infty(\RPlus),\ r\geq0).
\end{equation}
For every $j$ in $\N$, define
$\B_j\colon L_\infty(\RPlus)\to L_\infty(\RPlus)$ by
\begin{equation}
\label{promedio-Bj}
(\B_j g)(r)
\eqdef
\int_{r}^{+\infty}
(\B_{j-1}g)(u)\enumber{r-u}\,\dd u
\qquad(g\in L_\infty(\RPlus),\ r\geq0).
\end{equation}
In~\cite{Grudsky-Vasilevski},
these operators are defined on a wider function class,
but here we restrict them to $L_\infty(\RPlus)$.
Since $\int_r^{+\infty}\enumber{r-u}\,\dd u=1$,
the values of $\B_j$ belong to $L_\infty(\RPlus)$.


\begin{lemma}
\label{lem:TBj}
For every $j$ in $\No$
and every $g$ in $L_\infty(\RPlus)$,
\begin{equation}
\label{eq:tauLjga=gabj}
\tau_{L}^{j}\gamma_{g}
= 
\gamma_{(\B_{j}g)\circ\,\operatorname{sq}},
\end{equation}
where
$\operatorname{sq}\colon\RPlus\to\RPlus$
is defined by
$\operatorname{sq}(x)\eqdef x^{2}$.
\end{lemma}

\begin{proof}
Obviously, \eqref{eq:tauLjga=gabj} holds for $j=0$.
Supposing that $j\in\N$ and~\eqref{eq:tauLjga=gabj} holds for $j-1$ instead of $j$, let us prove it for $j$.
Let $g\in L_\infty(\RPlus)$ and $n\in\No$.
By the induction hypothesis,
~\eqref{gamma-radial-fock} and~\eqref{eq:tauLjga=gabj},
\begin{align*}
(\tau_{L}^{j}\gamma_{g})(n)
&=
(\tau_L^{j-1}\gamma_g)(n+1)
=
\gamma_{(\B_{j-1}g)\circ\,\operatorname{sq}}(n+1)
\\[0.5ex]
&=
\frac{1}{(n+1)!}\int_{0}^{+\infty}
(\B_{j-1} g)(r)\,\enumber{-r}r^{n+1}\,\dd r.
\end{align*}
Using the fundamental theorem of calculus
(in the context of Lebesgue integrals),
it is easy to see that for $j\ge 1$,
the function 
$f(r)\eqdef (\B_j g)(r) \enumber{-r}$
is differentiable almost everywhere in $\RPlus$, and
\begin{equation}
\label{eq:deriv_Bj}
f'(r)
= -(\B_{j-1} g)(r) \enumber{-r}.
\end{equation}
Using~\eqref{eq:deriv_Bj} and integrating by parts, we get
\begin{align*}
(\tau_{L}^{j}\gamma_{g})(n)
&=
-\frac{1}{(n+1)!}\int_{0}^{+\infty}
f'(r) r^{n+1}\,\dd r
=\frac{1}{n!} \int_{0}^{+\infty}
f(r) r^n\,\dd r
\\[0.5ex]
&=
\frac{1}{n!} \int_{0}^{+\infty}
(\B_j g)(r) r^n \enumber{-r}\,\dd r
=
(\gamma_{(\B_{j}g)\circ\,\operatorname{sq}})(n).
\qedhere
\end{align*}
\end{proof}

\begin{proposition}
\label{prop:gamma-inv-tl}
$\gamma(L_{\infty}(\RPlus))$
is an invariant subspace of $\tau_{L}$.
\end{proposition}

\begin{proof}
Follows from Lemma~\ref{lem:TBj}.
\end{proof}

\begin{proposition}\label{Prop:taulNsig-gamma}
Let $\sigma\in\SQRTO$. Then, for any $\eps>0$,
there exist $N$ in $\N$ and $g$ in $L_{\infty}(\RPlus)$ such that
\[
\bigl\|\tau_{L}^{N}\sigma-\gamma_{g}\bigr\|<\eps.
\]
\end{proposition}

\begin{proof}
  Let $\sigma\in\SQRTO$.  Given any $\varepsilon>0$, by \cite[Proposition 6.8]{Esmeral-Maximenko-2016},
 there exist
$h\in\,L_{\infty}\big(\RPlus\big)$ and $N\in\N$ such that 
\begin{equation}\label{eq:aux-sigma-gamma-h}
\displaystyle\sup_{n\geq N}\big|\sigma(n)-\gamma_{h}(n)\big|<\varepsilon.
\end{equation} 
Let $g=\big(\mathcal{B}_{N}h\big)\circ \operatorname{sq}$. Then $g$ belongs to $L_{\infty}(\RPlus)$ and $\gamma_{g}=\tau_{L}^{N}\gamma_{h}$  by Lemma \ref{lem:TBj}. Therefore,  by \eqref{eq:aux-sigma-gamma-h} we have
\begin{align*}\label{eq:n>Nsupsigma}
\big\|\tau_{L}^{N}\sigma-\gamma_{g}\big\|_{\infty}&=\big\|\tau_{L}^{N}\sigma-\tau_{L}^{N}\gamma_{h}\big\|_{\infty}=\sup_{m\in\No}\big|\sigma(N+m)-\gamma_{h}(m+N)\big|\\
&=\sup_{n\geq N}\big|\sigma(n)-\gamma_{h}(n)\big|<\varepsilon.\qedhere
\end{align*}
\end{proof}

The following Proposition is inspired by~\cite[Theorem~3.1]{Grudsky-Vasilevski-2}.

\begin{proposition}
\label{prop:gamma_which_vanishes_from_certain_point}
Let $g\in\,L_{\infty}(\RPlus)$ and $p\in\No$
such that $\gamma_{g}(n)=0$ for all $n\in\No$ with $n\ge p$.
Then, $g(x)=0$ a.e. in $\RPlus$. 
\end{proposition}

\begin{proof}
For every $k$ in $\No$, by the assumption, we have
\[
0=\gamma_{g}(k+p)
=\dfrac{1}{(k+p)!}
\int_{0}^{\infty} f(r)\,r^k\,\enumber{-r/2}\,\dd r,
\]
where
\[
f(r) \eqdef g(\sqrt{r})\,r^{p}\,\enumber{-r/2}.
\]
Notice that $f\in L_2(\RPlus)$.
By the linear property of the integral,
we conclude that for every $k$ in $\No$,
\[
\int_{0}^{+\infty} f(r) L_k(r) \enumber{-r/2}\,\dd r=0,
\]
where $L_k$ is the Laguerre polynomial of degree $k$.
We recall the definition of Laguerre polynomials below,
see \eqref{Laguerre-poly} and~\eqref{Laguerre-poly-explicit}.
The sequence of the ``Laguerre functions''
$r \mapsto L_k(r)\enumber{-r/2}$
form an orthogonal basis of $L_2(\RPlus)$
(see, e.g., \cite[Theorem~5.7.1]{Szego}).
Therefore, $f=0$ a.e.
\end{proof}

\section{Laguerre polynomials and eigenvalues sequences of radial Toeplitz operators}\label{Laguerre+radial-T}
In this section,
for any $m\in\No$,
we construct a sequence of defining symbols
$\left(a_{m,\xi}\right)_{\xi\in\No}$
in $C_{0}(\RPlus)$,
such that the following uniform approximation holds:
\[
\lim_{\xi\to\infty}
\bigl\|\gamma_{a_{m,\xi}}-\delta_{m}\bigr\|_{\infty}=0.
\]
As a by-product of this fact,  we conclude that
$\gamma(C_{0}(\RPlus))$ is dense in $c_{0}(\No)$, where  $\gamma(C_{0}(\RPlus))$ is given by~\eqref{eq:def-Gamma0}.

\begin{remark}
Let $m\in\No$.
By Proposition~\ref{prop:gamma_which_vanishes_from_certain_point},
the inverse eigenvalues problem ``$\delta_m=\gamma_g$''
has no exact solution $g$ in the class $L_\infty(\RPlus)$.
In other words,
$\delta_m\notin\gamma(L_\infty(\RPlus))$.
\end{remark}

Recall the Rodrigues formula and the explicit formula for the Laguerre polynomials
(see, e.g.,
\cite[(5.1.5) and (5.1.6)]{Szego}
or \cite[Sections 3.5 and 3.6]{Doman}):
\begin{equation}
\label{Laguerre-poly}
L_{n}(x)
=
\dfrac{1}{n!}\enumber{x}
\dfrac{\dd^{n}}{\dd x^{n}}
\left(\enumber{-x}x^{n}\right),
\end{equation}
\begin{equation}
\label{Laguerre-poly-explicit}
L_n (x) = \sum_{k=0}^n (-1)^k \binom{n}{k}
\frac{x^k}{k!}. 
\end{equation}
We notice that the functions $x \mapsto \enumber{-x}x^{n}$
 appear not only in \eqref{Laguerre-poly} but also  as weights for the integral
in~\eqref{gamma-radial-fock}.
Therefore, if defining symbols are given in terms of Laguerre polynomials, then the corresponding eigenvalue sequence is easy to calculate.
Inspired by this fact and by some ideas from~\cite{Barrera-Maximenko-Ramos},
we have invented the following special family of defining symbols.

\begin{definition}
\label{def:special_generating_symbols}
Let $m\in\No$ and $\xi\in\N$ with $\xi\geq2$.
Define $a_{m,\xi}\colon\RPlus\to\Real$,
\begin{equation}
\label{eq:seq-am}
a_{m,\xi}(x)
\eqdef
(-1)^{m}
\xi^{m+1}
\enumber{-(\xi-1)x^{2}}
L_{m}(\xi x^{2}).
\end{equation}
\end{definition}

\begin{lemma}
\label{lema-axim-fock}
Let $m\in\No$ and $\xi\in\N$ with $\xi\geq2$.
Then, $a_{m,\xi}\in C_0(\RPlus)$,
$\gamma_{a_{m,\xi}}\in \Gamma_0$,
and
\begin{equation}
\label{eq:gamma-a-mxi}
\gamma_{a_{m,\xi}}(n)
=
\begin{cases}
0, & n < m;
\\
\displaystyle\binom{n}{m}
\frac{1}{\xi^{n-m}},
& n \ge m.
\end{cases}
\end{equation}
\end{lemma}

\begin{proof}
By~\eqref{Laguerre-poly-explicit},
our function $a_{m,\xi}$
is a linear combination
of functions of the following form,
with $k\in\{0,\ldots,m\}$:
\begin{equation}
\label{eq:power_by_exp_special}
x \mapsto (-1)^{m}\xi^{m+1}x^{k}\enumber{-(\xi-1)x^{2}}.
\end{equation}
It is easy to see that all these functions belong to $C_0(\RPlus)$.
Therefore, $a_{m,\xi}\in C_0(\RPlus)$
and $\gamma_{a_{m,\xi}}\in \Gamma_0$.

Let $n\in\No$.
To deduce an explicit formula for $\gamma_{a_{m,\xi}}(n)$,
we start with~\eqref{gamma-radial-fock}
and apply the change of variables $r=\xi x$:
\begin{align*}
\gamma_{a_{m,\xi}}(n)
&=
\dfrac{\xi^{m+1}}{n!}
\int_{0}^{+\infty}\hspace{-3pt}
\enumber{-\xi x}
L_{m}(\xi x)x^{n}\,\dd x
\\
&=
\dfrac{(-1)^{m}}{n!\,\xi^{n-m}}
\int_{0}^{+\infty}\hspace{-3pt}
\enumber{-r}L_{m}(r)r^{n}\,\dd r.
\end{align*}
The last integral is related to the expansion of the function $r^n$ in terms of the Laguerre polynomials $L_m(r)$.
By~\cite[3.6.5]{Doman},
\begin{equation}
\label{for:lagxn}
\int_{0}^{+\infty}
\enumber{-r}L_m(r)r^n\,\dd r
=\begin{cases}
0, & m>n
\\
\displaystyle \dfrac{(-1)^{m}(n!)^{2}}{(n-m)!\,m!},
& m\leq n.
\end{cases}
\end{equation}
Hence, \eqref{eq:gamma-a-mxi} holds.
\end{proof}

\begin{theorem}
\label{thm-gamma-delta}
Let $m\in\No$. Then,
\[
\lim_{\xi\to\infty}
\bigl\|\gamma_{a_{m,\xi}}-\delta_{m}\bigr\|_{\infty}=0.
\]
As a consequence,
$\delta_{m}\in\overline{\gamma(C_{0}(\RPlus))}$.
\end{theorem}

\begin{proof}
We suppose that $\xi\in\N$ be such that $\xi\geq \dfrac{m+2}{2}$.
For brevity, we denote $\gamma_{a_{m,\xi}}$ by $f_\xi$.
By Lemma~\ref{lema-axim-fock},
the first $m+1$ values of the sequence $f_\xi$
coincide with the corresponding values of the sequence $\delta_m$:
\[
f_\xi(0)
=f_\xi(1)
=\cdots
=f_\xi(m-1)=0,\qquad
f_\xi(m)=1.
\]
For $n\geq m+1$, we have
\begin{align*}
f_\xi(n)-f_\xi(n+1)
&=
\binom{n}{m} \frac{1}{\xi^{n-m}}
-\frac{n+1}{n+1-m}\,\binom{n}{m} \frac{1}{\xi^{n+1-m}}
\\
&=
\frac{1}{(n+1-m)\xi^{n+1-m}}
\binom{n}{m}
\left((n+1-m)\xi - (n+1)\right)
\\
&\geq
\binom{n}{m}
\dfrac{2}{(n+1-m)\,\xi^{n-m+1}}\left(\,\xi -\dfrac{m+2}{2}\right)
\geq
0.
\end{align*}
Therefore, the sequence $(f_\xi(n))_{n=m+1}^\infty$
is decreasing, and
\begin{align*}
\bigl\|\gamma_{a_{m,\xi}}-\delta_{m}\bigr\|_{\infty}
&=
\sup_{n\geq m+1}
\bigl|f_\xi(n)-\delta_{m}(n)\bigr|
=
\sup_{n\geq m+1} f_\xi(n)
\\
&=
f_\xi(m+1)
=\binom{m+1}{m} \frac{1}{\xi^{m+1-m}}
=\frac{m+1}{\xi}.
\end{align*}
The last expression tends to $0$ as $\xi$ tends to $\infty$.
\end{proof}

We denote by $c_{00}(\No)$
the vector space of complex sequences of finite support.
In other words,
\[
c_{00}(\No)
= \bigl\{\sigma\in\Complex^{\No}\colon\
\exists N\in\No\quad \forall n\ge N\quad \sigma(n)=0\bigr\}.
\]
It is well known that $c_{00}(\No)$ is a dense subset of $c_0(\No)$.

In the following corollary,
we use the obvious linear property of the function $\gamma\colon L_{\infty}\to\ell_{\infty}(\No)$:
\begin{equation}
\label{eq:gamma_linear_property}
\gamma_{g+h}=\gamma_{g}+\gamma_{h},\qquad \gamma_{ch}=c\,\gamma_{h}\qquad (g,h\in L_{\infty}(\RPlus),\ c\in\Complex). 
\end{equation}

\begin{corollary}
\label{Cor:gamma-dense-c00}
$c_{00}(\No)$ is contained in the closure of $\gamma(C_{0}(\RPlus))$.
\end{corollary}

\begin{proof}
Let $\sigma\in c_{00}(\No)$.
We choose $N$ in $\No$ such that $\sigma(n)=0$ for all $n\ge N$.
Then, $\sigma$ can be written as the following linear combination of basic sequences:
\[
\sigma = \sum_{k=0}^{N-1} \sigma(k) \delta_k.
\]
For every $\xi$ in $\N$ with $\xi\geq2$,
we define $u_\xi$ as the linear combination of the functions $a_{k,\xi}$ with the same coefficients:
\begin{equation}
\label{eq:b_xi_def}
u_\xi
\eqdef 
\sum_{k=0}^{N-1}\sigma(k)a_{k,\xi}.
\end{equation}
More explicitly,
\begin{equation}
\label{eq:b_xi_explicit}
u_\xi(x)
=
\enumber{-(\xi-1)x^{2}}\sum_{k=0}^{N-1}(-1)^{k}\sigma(k)\xi^{k+1}L_{k}(\xi x^{2}).
\end{equation}
Then, $u_\xi\in C_0(\RPlus)$ and
\[
\bigl\|\gamma_{u_{\xi}}-\sigma\bigr\|_{\infty}
\leq
\bigl\|\sigma\bigr\|_\infty
\sum_{k=0}^{N-1}\bigl\|\gamma_{a_{k,\xi}}-\delta_{k}\bigr\|_{\infty}.
\]
The last expression tends to $0$ by Theorem~\ref{thm-gamma-delta}.
\end{proof}

\begin{theorem}
\label{thm:gamma-dense-c0}
$\gamma(C_{0}(\RPlus))$ is a dense subset of $c_{0}(\No)$.
\end{theorem}

\begin{proof}
By Corollary~\ref{Cor:gamma-dense-c00},
$c_{00}(\No)\subseteq\overline{\gamma(C_{0}(\RPlus))}$.
Since $c_{00}(\No)$ is a dense subset of $c_0(\No)$,
we conclude that
\[
c_0(\No)
=\overline{c_{00}(\No)}
\subseteq \overline{\gamma(C_{0}(\RPlus))}
\subseteq c_0(\No).
\]
Here is a more detailed reasoning.
Given a sequence $\sigma$ in $c_0(\No)$
and a number $\eps>0$,
we find $N$ in $\No$ such that $|\sigma(k)|<\eps/2$
for every $k$ in $\No$ with $k\ge N$.
For every $\xi$ in $\N$ with $\xi\geq2$,
we construct $u_\xi$ by~\eqref{eq:b_xi_def}.
Notice that we only use the first $N$
elements of the sequence $\sigma$
as coefficients in~\eqref{eq:b_xi_def}.
Reasoning as in the proof of Corollary~\ref{Cor:gamma-dense-c00},
we see that
\[
\lim_{\xi\to\infty} \max_{0\le k<N} |u_\xi(k) - \sigma(k)| = 0.
\]
On the other hand, reasoning as the proof of Theorem \ref{thm-gamma-delta},  we get
\[
\lim_{\xi\to\infty} \sup_{k\ge N} |u_\xi(k)| = 0.
\]
Therefore, for $\xi$ large enough,
\[
\sup_{k\ge N} |u_\xi(k) - \sigma(k)|
\le
\sup_{k\ge N} |u_\xi(k)| + \sup_{k\ge N} |\sigma(k)|
<
\frac{\eps}{2}+\frac{\eps}{2}
= \eps.
\]
Joining the cases $k<N$ and $k\ge N$, we conclude that
$\|u_\xi - \sigma\|_\infty <\eps$ for $\xi$ large enough.
\end{proof}

We denote by $C_{\lim}(\RPlus)$
the vector space consisting of all continuous functions
$\RPlus\to\Complex$
that have a finite limit at $+\infty$.
This is a Banach space with respect to the supremum-norm.
It can be identified with $C([0,+\infty])$.
We consider $C_{\lim}(\RPlus)$
as a subspace of $L_\infty(\RPlus)$. 

\begin{theorem}
\label{thm:gamma-dense-c}
$\gamma(C_{\lim}(\RPlus))
$
is a dense subset of $c(\No)$.
\end{theorem}

\begin{proof}
Let $\mathbf{1}_{\No}$ be the constant $1$ sequence
(defined on $\No$)
and $\mathbf{1}$ be the constant $1$ function defined on $\RPlus$.
The next simple property follows directly from~\eqref{gamma-radial-fock}:
\begin{equation}
\label{eq:gamma_of_constant_function}
\gamma_{\mathbf{1}}=\mathbf{1}_{\No}.
\end{equation}
Let $\sigma\in c(\No)$.
We denote the limit of $\sigma$ by $p$.
Then, the sequence
$(\sigma(n)-p \mathbf{1}_{\No})_{n\in\No}$ belongs to $c_0(\No)$.
By Theorem~\ref{thm:gamma-dense-c0},
there exists a sequence of functions
$(u_\xi)_{\xi\in\N}$ such that
$u_\xi\in C_0(\RPlus)$ for every $\xi$, and
\begin{equation}
\label{eq:approximation_of_the_sequence_after_restring_the_limit}
\lim_{\xi\to+\infty}
\bigl\|\gamma_{u_\xi}
-\left(\sigma-p \mathbf{1}_{\No}\right)\bigr\|_{\infty}
= 0.
\end{equation}
For each $\xi$, we define $v_\xi\colon\No\to\Complex$ by
\[
v_\xi \eqdef u_\xi + p \mathbf{1}.
\]
Then, by~\eqref{gamma-radial-fock}
and~\eqref{eq:gamma_linear_property},
\[
\gamma_{v_\xi}(n)
= \gamma_{u_\xi}(n) + p \mathbf{1}_{\No}.
\]
Thus, \eqref{eq:approximation_of_the_sequence_after_restring_the_limit}
can be rewritten as
\[
\lim_{m\to+\infty}
\bigl\|\gamma_{v_\xi}-\sigma\bigr\|_{\infty}
=0.
\qedhere
\]
\end{proof}
\begin{remark}
Now, using results from this paper,
we give a more constructive and clear proof
of the main density result from~\cite{Esmeral-Maximenko-2016}:
\[
\overline{\gamma(L_\infty(\RPlus))}=\SQRTO.
\]
Let $\sigma\in\SQRTO$ and $\eps>0$.
By Proposition~\ref{Prop:taulNsig-gamma},
there exist $N$ in $\N$ and $g$ in $L_\infty(\RPlus)$ such that
\begin{equation}
\label{eq:final-remark-1}
\|\tau_{L}^{N}\sigma-\gamma_{g}\|_{\infty}<\dfrac{\eps}{2}.
\end{equation}
Let $x\eqdef \sigma - \tau_L^N \sigma$.
By Proposition~\ref{eq:sequence_minus_its_shift},
$x\in c_0(\No)$.
By Theorem \ref{thm:gamma-dense-c0},
there exists $h$ in $C_0(\RPlus)$ such that
\begin{equation}
\label{eq:final-remark-2}
\|x-\gamma_{h}\|_{\infty}<\dfrac{\eps}{2}.
\end{equation}
Finally, we set $f \eqdef g+h$.
Then, $f\in L_\infty(\RPlus)$ and,
by~\eqref{eq:final-remark-1} and~\eqref{eq:final-remark-2},
\begin{align*}
\|\sigma-\gamma_f\|_{\infty}
&=
\|(\tau_L^N \sigma + x) - (\gamma_g + \gamma_h)\|_\infty
\\
&\leq
\|\tau_{L}^{N}\sigma - \gamma_{g}\|_{\infty}
+ \|x-\gamma_{h}\|_{\infty}
<
\dfrac{\eps}{2}+\dfrac{\eps}{2}=\eps.
\end{align*}  
\end{remark}

\section*{Declarations}

\subsection*{Funding}

The first author has received research support for the project PRY-63 from Universidad de Caldas, Colombia.
The research of the second author has been supported
by Instituto Polit\'{e}cnico Nacional, Mexico
(Proyectos IPN-SIP)
and by SECIHTI, Mexico
(Proyecto ``Ciencia de Frontera'' FORDECYT-PRONACES/61517/2020).
 
\subsection*{Data Availability}
This theoretical research does not use any data set.

\end{document}